\documentclass[a4paper]{amsart}
\usepackage{amsmath,amssymb,amsthm}
\usepackage{amscd}
\usepackage{longtable}
\theoremstyle{definition}
\numberwithin{equation}{section}
\newtheorem{thm}{Theorem}[section]

\newtheorem{exa}[thm]{Example}
\newtheorem{prop}[thm]{Proposition}
\newtheorem{cor}[thm]{Corollary}
\newtheorem{lem}[thm]{Lemma}
\newtheorem{rem}[thm]{Remark}

\newcommand{\bracket}[1]{{\langle {#1} \rangle}}
\def\C{\mathbb C}   \def\Z{\mathbb Z}
\def\P{\mathbb P} 
\def\s{\sigma}

\begin{document}
\title[Bi-canonical representations of Enriques surfaces]{Bi-canonical representations of finite automorphisms acting on Enriques surfaces}
\author{Hisanori Ohashi}
\begin{abstract}
We classify the bi-canonical representations of finite automorphisms on  
Enriques surfaces. There are three types of non-trivial cases and 
examples are given explicitly by Horikawa models. 
In particular, finite non-semi-symplectic automorphisms exist 
only in orders 4 and 8. One corollary is that any finite cyclic subgroup 
in the automorphism group of an Enriques surface has order 1, 2, 3, 4, 5, 6, 8.
Moreover, for two of the three types, a uniqueness theorem for maximal-dimensional 
families is given.
\end{abstract}

\maketitle

\section{Introduction}

An automorphism $\sigma$ acting on an Enriques surface $S$ is 
called {\em{semi-symplectic}} if it acts trivially on the space of global bi-2-forms $H^0(\mathcal{O}_S(2K_S))$. It is {\em{non-semi-symplectic}} if this action 
is non-trivial. If $\sigma$ acts on $H^0(\mathcal{O}_S(2K_S))\simeq \C$ 
by a primitive $I$-th root of unity, we say $\sigma$ has {\em{index}} $I=I(\s)$.
These are the notions parallel to symplectic and non-symplectic
automorphisms of $K3$ surfaces, in which case we look at the action on
the space of global 2-forms $H^0(\mathcal{O}_X(K_X))$ for a $K3$ surface $X$.

Previously, examples of non-semi-symplectic automorphisms
appeared only implicitly in the works \cite{MN} and \cite{kondo86}. 
In \cite[Example 3]{MN}, a 1-dimensional family of Enriques surfaces 
with an automorphism of order 4 is exhibited.
Using the fact that it acts numerically trivially on cohomology,  
Mukai and Namikawa have computed its action on the 
transcendental lattice (of the covering $K3$ surface) as 
a $4\times 4$ matrix $A_0$ \cite[p390, (*)]{MN}. Since the eigenpolynomial of $A_0$ is 
$(x^2+1)^2$, this is an example of a non-semi-symplectic automorphism of order 4 and 
index 2. Another example, Kondo's type III surface \cite[Example III]{kondo86}, is the one 
constructed by the {\em{special Lieberman involution}} \cite[Remark 3.3.3]{kondo86}, 
namely the $K3$-cover is the Kummer 
surface of $A=E_{\sqrt{-1}}\times E_{\sqrt{-1}}\ \ (E_{\sqrt{-1}}=\C/\Z+\Z\sqrt{-1})$ 
and the Enriques involution comes from the 
Lieberman involution using the $2$-torsion point $((1+\sqrt{-1})/2,(1+\sqrt{-1})/2)$. The automorphism $(\sqrt{-1}, 1)$ 
of $A$ descends to the Kummer surface and the Enriques quotient, hence it gives an
example of a non-semi-symplectic automorphism of order 4 and 
index 2. (We remark that, therefore the automorphism group of \cite[Example III]{kondo86}
cannot be generated by involutions. It 
is in fact a {\em{non-split}} extension $(\Z/2)^4D_8$).

In this paper, we study finite non-semi-symplectic automorphisms of Enriques surfaces
in general.
We give three examples below in Examples \ref{(4,2)}, \ref{(8,4)}, \ref{(8,2)}. 
They are described using the birational projective model of elliptic Enriques surfaces
doubly covering $\P^1\times \P^1$
by Horikawa \cite{horikawa1}. See also Section \ref{models}. 
Let $y,z$ denote the inhomogeneous coordinates of two projective lines.

\begin{exa}\label{(4,2)}
Let $S$ be the Enriques surface defined by the equation
\[\begin{split}
w^2=z(A(y^4z^2-z^2)+ & B(y^4z-z^3)+C(y^4-z^4)+D(y^3z^2-yz^2)\\
&+E(y^3z-yz^3)+F(y^2z-y^2z^3)),
\end{split}\]
where $A,\dots,F\in \C$ are parameters.
It is invariant under the automorphism 
\[\sigma_1 \colon (w,y,z)\mapsto \left(\frac{\sqrt{-1}w}{y^2z^3}, \frac{1}{y},\frac{1}{z}\right)\]
of order $4$. The global bi-2-form on $S$ is given by 
$\Omega_S=z(dy\wedge dz/w)^{\otimes 2}$ (see Section \ref{models}),
which is negated by $\sigma_1$. Therefore $\sigma_1$ is a non-semi-symplectic automorphism
of order $4$ and index $2$. 
\end{exa}
\begin{exa}\label{(8,4)}
The second example arises as a subsystem of the previous one. 
Let $S$ be the Enriques surface defined by the equation 
\[\begin{split}
w^2=z(A(y^4z^2+\sqrt{-1}z^4-z^2- & \sqrt{-1}y^4)+ 
B(y^4z+\sqrt{-1}y^2z^3-z^3-\sqrt{-1}y^2z)\\
&+D(y^3z^2+\sqrt{-1}yz^3-yz^2-\sqrt{-1}y^3z)),
\end{split}\]
where $A,B,D\in \C$ are parameters as in Example \ref{(4,2)}, i.e.,
we made the specializations $C=-\sqrt{-1}A, E=-\sqrt{-1}D, F=-\sqrt{-1}B$. 
This subfamily is invariant under the automorphism 
\[\sigma_2 \colon (w,y,z)\mapsto \left(\frac{\zeta_8 y^3 w}{z^4}, \frac{y}{z},\frac{y^2}{z}\right),\]
where $\zeta_8=(1+\sqrt{-1})/\sqrt{2}$ is the primitive eighth root of unity.
Note that $\sigma_2^2=\sigma_1$, hence $\sigma_2$ has order $8$.
It satisfies $\sigma_2^* \Omega_S = -\sqrt{-1}\Omega_S$, hence this is an 
example of an automorphism of order $8$ and index 4.
\end{exa}
\begin{exa}\label{(8,2)}
Let $S$ be the Enriques surface defined by the equation 
\[w^2=z(Ay^4z^2+B(y^4+z^4)+C(y^3z-\sqrt{-1}yz^3)+Dz^2),\]
where $A,\dots,D\in \C$ are parameters. 
This equation is invariant under the automorphism 
\[\sigma_3 \colon (w,y,z) \mapsto \left( \frac{wy^3}{z^3}, \sqrt{-1}y, \frac{y^2}{z} \right)\]
of order 8. It acts on $\Omega_S$ by negation, hence 
$\sigma_3$ is an automorphism of order 8 and index 2. 
\end{exa}
Our classification of bi-canonical representations goes as follows.
\begin{thm}\label{clsn}
Let $\s$ be a finite non-semi-symplectic automorphism of an Enriques surface $S$.
Then the order $\mathrm{ord}(\s)$ and the index $I=I(\s)$ takes only one of the following 
values.
\[(\mathrm{ord}(\s), I(\s))=(4,2),\ (8,4),\ (8,2).\]
\end{thm}
The existence of these values are assured by the examples above.

A {\em{family of automorphisms}} is a smooth projective family $\pi \colon \mathcal{X}
\rightarrow \mathcal{S}$ with an automorphism $\Phi\in \mathrm{Aut}(\mathcal{X})$ 
such that $\pi \circ \Phi=\Phi$. We regard $\Phi|_{\pi^{-1}(s)}\in \mathrm{Aut}(\mathcal{X}_s)$ 
as a varying family of automorphisms.
From this perspective, 
Example \ref{(4,2)} is not the only family with non-semi-symplectic automorphisms
with $(\mathrm{ord}(\s), I(\s))=(4,2)$. In fact, the cohomology representation of 
\cite[Example 3]{MN} is numerically trivial, while it is not for Example \ref{(4,2)}
by Lemma \ref{18}.
The next theorem studies the possible dimensions (moduli numbers) of effectively parametrized
families of automorphisms. Note that the order and index are invariant in a family.

\begin{thm}\label{dimensions}
Let $\s$ be a finite non-semi-symplectic automorphism of an Enriques surface $S$.
Assume that $\s$ is in a family of automorphisms. 
Then the moduli number of the family is bounded from above by the following value $m$.
\begin{center}
\begin{tabular}{c||c|c|c}
$(\mathrm{ord}(\s), I(\s))$ & $(4,2)$ & $(8,4)$ & $(8,2)$ \\ \hline
maximal $m$ & $5$ & $2$ & $2$ \\
\end{tabular}
\end{center}
Conversely, Examples \ref{(4,2)}, \ref{(8,4)}, \ref{(8,2)}
have moduli numbers $5,2,2$ respectively. 
Therefore, we may say that $m$ is the {\em{maximal}} number of moduli
and the examples above provide a {\em{maximal family}} for each case.
\end{thm}
\begin{rem}
There is a subtlety in terminologies here; in fact, \cite[Example 3]{MN} can be 
obtained as a specialization of Example \ref{(4,2)} by 
\[A=1,\ B=-C-1,\ D=E=0,\ F=-C+1.\]
In this sense, families like these two are sometimes considered equivalent.
But, the fact is that the specialization produces rational double points 
on surfaces and the automorphism does not lift to the simultaneous resolution
biregularly. It results in the different cohomological representations 
and different topologies of fixed loci. By this reason, we adopt the stronger notion of 
a family, namely a family equipped with automorphisms in this paper.
\end{rem}

As we will see in the proof of Theorem \ref{dimensions}, if $(S,\s)$ is a very general member 
of an $m$-dimensional family of non-semi-symplectic automorphisms with 
values $(\mathrm{ord}(\s), I(\s))=(4,2),\ (8,4)$, then its $K3$-cover has 
Picard number 10. (In case $(\mathrm{ord}(\s), I(\s))=(8,2)$ it is 16.)
It is well-known that Enriques surfaces and their $K3$-covers do not have any
smooth rational curves if the $K3$-cover has Picard number 10. 
This strong fact allows us to prove that our Examples \ref{(4,2)} and \ref{(8,4)} are
the {\em{unique}} maximal families as follows.
\begin{thm}\label{chrn}
Let $\s$ be a non-semi-symplectic automorphism of an Enriques surface $S$ with 
$(\mathrm{ord}(\s), I(\s))=(4,2)$ or $(8,4)$. 
If the covering $K3$ surface $X$ of $S$ has Picard number $\rho=10$, then
$(S, \s)$ is a member of Examples \ref{(4,2)} or \ref{(8,4)} respectively.
\end{thm}
We will not go further on the questions such as, 
the number of families with small moduli numbers for $(\mathrm{ord}(\s), I(\s))=(4,2)$, or
the similar characterization theorem in the remaining case $(\mathrm{ord}(\s), I(\s))=(8,2)$.
This situation is quite similar to that of primitively non-symplectic, 
or non-primitively non-symplectic automorphisms of $K3$ surfaces.

One consequence of this paper is as follows.
\begin{cor}\label{orders}
An automorphism of finite order $n$ for some Enriques surface exists 
if and only if $n\in \{1,2,3,4,5,6,8\}$.
\end{cor}

The organization of the paper is as follows. In Section \ref{models}, 
we give supplementary explanations on Horikawa's models and prove 
Theorem \ref{dimensions}.
In Section \ref{classification} we prove Theorem \ref{clsn} and Corollary \ref{orders}.
Finally in Section \ref{characterization},
we prove Theorem \ref{chrn}.\\

\noindent {\bf{Acknowledgement:}} The main part of this work was done 
when the author stayed 
in Nagoya university in late March 2014. He is grateful to Professor S. Kondo for 
his hospitality. He is supported by JSPS Grant-in-Aid for Scientific Research (S) 22224001
and for Young Scientists (B) 23740010.\\

\noindent {\bf{Notation and Conventions:}}
We work over $\C$.
An {\em{Enriques surface}} is a smooth projective surface $S$ 
which satisfies $2K_S\sim 0, K_S\not\sim 0$ and $H^1(S,\mathcal{O}_S)=0$, where
$K_S$ is the canonical divisor class. The $K3$-cover $X$ of $S$ is defined by the 
spectrum $\mathrm{Spec} (\mathcal{O}_S\oplus \mathcal{O}_S(K_S))$ over $S$,
which is a $K3$ surface in the usual sense.

We denote by $U$ (resp. $U(2)$) the rank 2 lattice defined by the Gram matrix 
$\begin{pmatrix} 0 & 1 \\ 1 & 0 \end{pmatrix}$ (resp. $\begin{pmatrix} 0 & 2 \\ 2 & 0 \end{pmatrix}$).

For an automorphism $\varphi$ acting on a set $X$, we denote the subset of fixed elements
by $X^{\varphi}$. This applies both to geometric and arithmetic cases, namely $X$ may be 
a manifold or a lattice. The notation $\mathrm{Fix}(\varphi)$ is also used.

\section{On Horikawa models}\label{models}

Horikawa described elliptic Enriques surfaces of non-special
type in the following way. Let $y,z$ be the inhomogeneous coordinates of two rulings 
of a quadric surface $Q\simeq \P^1\times \P^1$. 
\begin{thm}\label{horikawamodel}(\cite[Theorem 4.1]{horikawa1})
Let $S$ be an Enriques surface which is not of special type. Then it is
birationally equivalent to a double covering $\overline{S}$ of $Q$ 
with branch locus $\Gamma_1 +\Gamma_2+B_0$, where 
$\Gamma_1=\{z=0\}$, $\Gamma_2=\{z=\infty\}$ and $B_0$ is of bidegree $(4,4)$
defined by a linear combination of the monomials
\[y^i z^j,\ 4\leqq i+2j\leqq 8,\ \  0\leqq i,j\leqq 4.\]
\end{thm}
\quad

If the coefficients of the linear combination are general,
the curve $B_0$ has tacnodes at the two points $(y,z)=(0,0)$ and $(\infty, \infty)$.
The tacnodal tangents coincide with $\Gamma_1$ and $\Gamma_2$ respectively.
The double cover $\overline{S}$ has two Gorenstein elliptic singularities 
of degree 1 over these tacnodes and each can be resolved by a single weighted blow-up of 
weights 3, 2, 1. Let $T$ be the smooth surface obtained this way. We can check that 
the strict transforms $\Gamma_{i,T}$ of $\Gamma_i$ become $(-1)$-curves
and we get a minimal Enriques surface $S$ by contracting them. Moreover,
let $\Omega_{\overline{S}}=z(dy\wedge dz/w)^{\otimes 2}$ be a bi-2-form 
on $\overline{S}$. 
By the resolution of singularities, the pullback $\Omega_T$ to $T$ of $\Omega_{\overline{S}}$
vanishes doubly along $\Gamma_{i,T}$, $i=1,2$. It has no other zeros nor poles over $T$, 
therefore we see that $\Omega_{\overline{S}}$ is the non-vanishing global bi-2-form 
characterizing the minimal Enriques surface $S$. \\

Our Examples \ref{(4,2)}, \ref{(8,4)}, \ref{(8,2)} all belong to this family.
In Example \ref{(4,2)}, the only isomorphisms between the members is the homothethy of 
the parameters $(A,B,C,D,E,F)$. 
Therefore it has $5$ moduli. The same holds for Example \ref{(8,4)}, hence
it has $2$ moduli. In Example \ref{(8,2)}, additionally $(y,z)\mapsto (\alpha y, \alpha z)$
induces a transformation $(A,B,C,D)\mapsto (\alpha^6A, \alpha^4B, \alpha^4C, \alpha^2D)$
of parameters for $\alpha\in \C^*$. Thus the third family has $2$ moduli.
This shows the latter part of Theorem \ref{dimensions}.

Let us prove the former part of Theorem \ref{dimensions}.
Let $(S,\s)$ be a non-semi-symplectic automorphism of order $n$ and 
index $I$.
Since the $K3$-cover $X$ is defined by the canonical divisor $K_S$, $\s$ lifts to an
automorphism $\varphi$ of $X$. Under our assumption that $n$ is even, 
the lift $\varphi$ has $\mathrm{ord}(\varphi)=n$. By the identification
\[H^0(X,\mathcal{O}_X(K_X))^{\otimes 2}\simeq H^0(S,\mathcal{O}_S(2K_S)),\]
$\varphi$ is a non-symplectic automorphism of index $2I$. 
By the general theory \cite{DK}, we have the period map from the family $\{(X,\varphi)\}$
to a period domain $\mathcal{D}$. By the Torelli theorem, this map is 
injective. Therefore it suffices to show that the period domain $\mathcal{D}$ 
has dimension at most $m$ in each case.

In case $n=2I$, $\varphi$ is a primitively non-symplectic automorphism. 
Let $\zeta_n$ be a primitive $n$-th root of unity.
The period domain of $(X,\varphi)$ is the 
projectivized $\zeta_{n}$-eigenspace
of $\varphi$ acting on the transcendental lattice $T_X\otimes \C$.
Its dimension is $m=\mathrm{rank} (T_X)/\phi (n)-1$, where $\phi$ is the Euler's function. 
Since we have $\mathrm{rank} (T_X)\leq 12$, 
we see that $\dim \mathcal{D}$ is at most 
$m=5$ (resp. $m=2$) for $n=4$ (resp. $n=8$). 
Moreover, if the equality holds, then $T_X$ is of rank $12$ and $\rho (X)=10$ follows.
This is the case for a very general member in an $m$-dimensional family.

In case $(n,I)=(8,2)$, the lift $\varphi$ is a non-symplectic automorphism of 
order $8$ and index 4. Let $\varepsilon$ be the 
covering involution of $X/S$. Since the eigenvalues of $\varphi$ on $T_X\otimes \C$ 
are 4-th roots of unity, the Picard number of $X$ is even. On the other hand, since $\varphi^2\varepsilon$ is 
a symplectic automorphism of order $4$, $X/\varphi^2\varepsilon$ has four $A_3$ 
and two $A_1$ singularities. Thus, by also including an ample divisor class,
we see $\rho (X)\geq 15$. 
Therefore $\rho(X)\geq 16$ and we obtain $\mathrm{rank} (T_X)\leq 6$. 
As in the previous cases, the period domain of $(X,\varphi)$ is a 
projectivized $\zeta_4$-eigenspace, hence has dimension at most $6/\varphi(4)-1=2$.
As before, if $(S,\s)$ is a very general member of a $2$-dimensional family, 
then its $K3$-cover has Picard number $16$. \\

Finally we note that the $K3$-cover of our Enriques surface $S$ is 
obtained as follows (\cite{horikawa1}). 
Let $X$ be the pullback of the double cover $\overline{S}\rightarrow Q$ by the morphism
\[\pi \colon Q\rightarrow Q,\ (Y,Z)\mapsto (y,z)=(YZ,Z^2).\]
If $f(y,z)=0$ is the defining equation of $B_0$ and $w^2=zf(y,z)$ that of $\overline{S}$, 
then $X$ has the equation 
\[W^2=g(Y,Z),\ \text{ where $g(Y,Z)=f(YZ,Z^2)/Z^4$ and $W=w/Z^3$.}\]
The divisor $\{g(Y,Z)=0\}\subset Q$ is of bidegree $(4,4)$ and is invariant under the 
involution $\iota\colon (Y,Z)\mapsto (-Y,-Z)$. The covering involution $\varepsilon$
of $X \rightarrow S$ is nothing but the lift of $\iota$, namely
$(W,Y,Z)\mapsto (-W,-Y,-Z)$. This is another, probably more familiar construction of the same 
Enriques surface from the paper
\cite{horikawa1}. We can rewrite our examples by using the $K3$-cover $X\rightarrow Q$ 
rather than $\overline{S}\rightarrow Q$. 
Our choice of the description in the introduction was because it is concise and direct.\\

\noindent {\bf{Example \ref{(4,2)} (resp. \ref{(8,4)}) bis:}} Let $B\colon \{g(Y,Z)=0\}$ be a divisor of bidegree 
$(4,4)$ on $Q$ satisfying the conditions
\begin{equation}\label{bis}
\begin{split}
 & g(-Y,-Z)=g(Y,Z),\ \text{ and }Y^4Z^4g(1/Y,1/Z)=-g(Y,Z). \\
(\text{resp. } & g(-Y,-Z)=g(Y,Z),\ \text{ and }Z^4g(1/Z,Y)=\sqrt{-1}g(Y,Z)).
\end{split}
\end{equation}
Then the double cover $X$ of $Q$ branched along $B$ is a $K3$ surface, whose 
affine equation is given by $W^2=g(Y,Z)$. It is equipped with 
a fixed-point-free involution $\varepsilon \colon (W,Y,Z)\mapsto (-W,-Y,-Z)$ 
if $g$ is general. The additional automorphism
\begin{equation*}
\begin{split}
& \varphi \colon (W,Y,Z)\mapsto (\sqrt{-1}W/Y^2Z^2, 1/Y,1/Z) \\
(\text{resp. } & \varphi \colon (W,Y,Z)\mapsto (\zeta_8 W/Z^2, 1/Z,Y)
\end{split}
\end{equation*}
commutes with $\varepsilon$, hence descends to the Enriques surface $S=X/\varepsilon$.
This gives a non-semi-symplectic automorphism of order $4$ and index 2 (resp. order 8 and 
index 4). \\

This form of the example will be used in Section \ref{characterization}.
We have a similar interpretation in Example \ref{(8,2)}, too. We omit the details.

\section{Classification}\label{classification}

In this section we prove Theorem \ref{clsn}.
Let $\s$ be an arbitrary finite automorphism of an Enriques surface $S$.
We start with recalling the following observations.
\begin{prop}\label{recall}
\begin{enumerate}
\item If $\mathrm{ord} (\s)$ is $2$ or an odd number, then it is automatically semi-symplectic.
\item Finite semi-symplectic automorphisms exist only up to order $6$.
\end{enumerate}
\end{prop}
\begin{proof} See \cite{MO_mathieu}.\end{proof}
\begin{prop}
For a finite non-semi-symplectic automorphism $\s$, the index $I(\s)$ is a power of $2$.
\end{prop}
\begin{proof}
Let $n=2^a \cdot b$ be the order of $\s$, decomposed into a $2$-power $2^a$ and 
an odd number $b$. By Proposition \ref{recall} (1), $\s^{2^a}$ is semi-symplectic. Therefore,
the bi-canonical representation is at most $2^a$-th root of unity.
\end{proof}

\noindent {\bf{Proof of Theorem \ref{clsn}.}}
\noindent {\bf{Step 1: $I=2$.}} First let us assume $\s$ has index $I=2$. Then
$\s^2$ is semi-symplectic. By Proposition \ref{recall} (2), 
$\s^2$ has order either $1,\dots$, or $6$.
However, Proposition \ref{recall} (1) excludes odd numbers.
The Examples \ref{(4,2)} and \ref{(8,2)} confirm the existence of $\mathrm{ord} (\s^2)=2,4$,
so in what follows we assume that $\tau = \s^2$ has order 6 and derive a contradiction.
Here, the semi-symplectic automorphism $\tau$ has a unique {\em{symplectic}} 
fixed point $P\in S$.
It is characterized among fixed points by the property that the local action 
$(d\tau)_P$ has $\det (d\tau)_P=1$. See the discussion before Proposition 3.8 in \cite{MO_mathieu}. Hence $\s$ also fixes $P$, $\s(P)=P$. 
From the original assumption $I=2$ we deduce that $(\det (d\s)_P)^2=-1$, but this is 
a contradiction because $(\det (d\s)_P)^2=\det (d{\tau})_P$.
Summarizing, we have only cases $(\mathrm{ord}(\s), I(\s))=(4,2)$ or $(8,2)$.\\

\noindent {\bf{Step 2: $I=4$.}} Let us assume $I=4$. By Step 1, we have $\mathrm{ord} (\s^2)$  
is either $4$ or $8$. Since Example \ref{(8,4)} confirms the existence of the former, 
let us assume $\mathrm{ord} (\s^2)=8$ and deduce a contradiction.
In this case, $\tau=\s^4$ is a semi-symplectic automorphism of order 4. 
Therefore it has exactly two symplectic fixed points $P_1,P_2$. 
Since $\s$ preserve $\{P_1,P_2\}$, $\s^2$ fixes both $P_i$. But then we have 
$(\det (d\s^2)_{P_i})^2=-1$ since $\s^2$ has index $2$, while $\det (d\tau)_{P_i}=1$ since
$P_i$ is a symplectic fixed point, a contradiction.
Thus we have only the case $(\mathrm{ord}(\s), I(\s))=(8,4)$ here.\\

\noindent {\bf{Step 3: $I=8$.}} 
In this case, we can only have $\mathrm{ord} (\s^2)=8$ by Step 2. 
Let $\tau = \s^8$ be the semi-symplectic power. Then the four isolated fixed points 
$P_i\ (i=1,\dots,4)$ of 
$\tau$ constitute exactly the symplectic fixed points. Since $\s$ preserves them,
$\s^4$ fixes all $P_i$. As before, we have 
$(\det (d\s^4)_{P_i})^2=-1$ since $\s^4$ has index $2$, while $\det (d\tau)_{P_i}=1$ since
$P_i$ is a symplectic fixed point, a contradiction. So, there are no finite non-semi-symplectic 
automorphisms of index 8, or higher.\\

Corollary \ref{orders} is an easy consequence since in the semi-symplectic case 
we have orders up to 6 by Proposition \ref{recall} (2), while for the non-semi-symplectic case
we have only orders 4 and 8 by Theorem \ref{clsn}.

\section{characterization}\label{characterization}

In this section we prove Theorem \ref{chrn}.\\

\noindent {\bf{Step 1:}} First let $S$ have an automorphism $\s$ with $(\mathrm{ord}(\s), I(\s))=(4,2)$. 
We assume that the $K3$-cover $X$ has Picard number $\rho (X)=10$. 
The lift $\varphi$ of $\s$ is an automorphism of order $4$ and index $4$, which 
we normalize so that $\varphi^*$ acts on $H^0(\mathcal{O}_X(K_X))$ by $\sqrt{-1}$. 
Let $\varepsilon$ be the covering transformation of the double covering $X/S$. 
\begin{lem}\label{18}
We have $H^2(X,\Z)^{\varphi^2}\simeq U(2)$ as lattices. Also we have 
$X^{\varphi^2}= C^{(9)}$, where $C^{(9)}$ is a smooth curve of genus $9$.
\end{lem}
\begin{proof}
In \cite{IO} we have classified involutions  
on Enriques surfaces into 18 types.
We apply it to $\s^2$. As in \cite{IO}, we denote by $K_-\subset H^2(X,\Z)$ the sublattice 
on which $\varepsilon$ and the non-symplectic involution $\varphi^2$ both act
by $-1$. Under our assumption $\rho (X)=10$, we get
\[\mathrm{rank} K_-\geqq \mathrm{rank} T_X=12.\]
Hence $\s^2$ belongs to No. 18. of \cite{IO}. The lemma follows.
\end{proof}
We use the following well-known lemma.
\begin{lem}
Let $X$ be a $K3$ surface with $\tau$ an involution. 
If $H^2(X,\Z)^{\tau}\simeq U(2)$, then one of the following holds.
\begin{itemize}
\item The quotient $X/\tau\simeq \P^1\times \P^1$. 
\item The quotient $X/\tau$ is isomorphic to the Hirzebruch surface $\mathbb{F}_2$ 
of degree 2 and $X$ has a smooth rational curve.
\end{itemize}
\end{lem}
Since the Picard number $\rho(X)$ is 10, $X$ has no smooth rational curves. Hence 
we see that $X/\varphi^2\simeq \P^1\times \P^1=Q$. The branch $B$ is a smooth curve of 
bidegree $(4,4)$ on $Q$, isomorphic to $C^{(9)}$.
The residue group $K_4=\bracket{\varphi, \varepsilon}/\varphi^2$ acts linearly on $Q$
preserving $B$. In what follows, we will determine the normal form of this action. 
\begin{lem}
(1) The group $K_4$ acts trivially on the Neron-Severi group $NS(Q)$. \\
(2) Non-trivial elements of $K_4$ are small, namely have only finitely many fixed points.
\end{lem}
\begin{proof}
Let $\tau$ denote either $\varphi$ or $\varepsilon \varphi$. Both
satisfy $\tau^2=\varphi^2$. By Lemma \ref{18}, the fixed point set $X^{\tau}$ is 
either (a) $=C^{(9)}$ or (b) a finite set consisting of $N$ points ($N\geqq 0$).
Applying the holomorphic Lefschetz theorem to $\tau$ gives the equality
\begin{equation*}
1-(\pm \sqrt{-1})= \begin{cases}
\frac{1-9}{1-(\pm \sqrt{-1})}-\frac{\pm 16\sqrt{-1}}{(1-(\pm \sqrt{-1}))^2} &(\text{in case (a)}),\\
\frac{N}{(1+1)(1\pm \sqrt{-1})} &(\text{in case (b)}),
\end{cases}
\end{equation*}
where $\pm$ is according to $\tau = \varphi$ or $\varepsilon \varphi$. 
The 1st line is untrue, hence we see that $X^{\tau}$ consists of $N=4$ points.
The topological Lefschetz formula then gives that $\tau$ acts on the 
lattice $H^2(X,\Z)^{\varphi^2}\simeq U(2)$ trivially.

The statement (1) follows since the residues of $\tau=\varphi, \varepsilon \varphi$ 
generate $K_4$ and $NS(Q)$ is pulled back onto $H^2(X,\Z)^{\varphi^2}$.
Similarly (2) follows since, the non-trivial representatives $\varphi^{\pm 1}, (\varepsilon \varphi)^{\pm1}$ and $\varepsilon, \varepsilon \varphi^2$ all have finitely many fixed points
by the above arguments.
\end{proof}
By (1) of the lemma, $K_4$ is a subgroup of $\mathrm{Aut}(\P^1)\times \mathrm{Aut}(\P^1)
\subset \mathrm{Aut}(Q)$
and by (2) both projections $K_4\rightarrow \mathrm{Aut}(\P^1)$ are injective. 
It is well-known that all subgroups of $PGL(2,\C)$ isomorphic to $K_4\simeq (\Z/2)^2$
are conjugate to each other, and one such is 
generated by involutions 
\begin{equation}\label{QQ}
\iota \colon (Y,Z)\mapsto (-Y,-Z)\ \text{ and } \varphi_1\colon  (Y,Z)\mapsto (1/Y,1/Z)
\end{equation}
where $Y,Z$ are inhomogeneous coordinates of projective lines. 
Therefore, $(S,\s)$ is exactly as in Example \ref{(4,2)} bis.\\

\noindent {\bf{Step 2:}} Next let $S$ have an automorphism $\s$ with $(\mathrm{ord}(\s), I(\s))=(8,4)$,
assuming that the $K3$-cover $X$ has Picard number $\rho (X)=10$. 
Let $\varphi$ be a lift of $\s$ to $X$, which has order 8 and index 8.
By the previous step, $(S,\s^2)$ comes from Example 1.1 bis., namely
$X$ is the double cover $W^2=g(Y,Z)$ where $g$ satisfies the conditions
\eqref{bis}
of Example 1.1 bis. The new automorphism $\varphi$ corresponds to an 
automorphism $\varphi_2$ of order 4 on $Q$ which satisfies $\varphi_2^2=\varphi_1$ 
in \eqref{QQ}.
Note that $PGL(2,\C)$ does not contain a subgroup isomorphic to $\Z/4\times \Z/2$, hence 
$\varphi_2$ exchanges the two rulings of $Q$. Now it is easy to see that such $\varphi_2$ is 
given only by $(Y,Z)\mapsto \pm (1/Z,Y)$. This is exactly the same as Example \ref{(8,4)} bis.
This completes the proof.

\end{document}